\documentclass[11pt]{amsart}

\usepackage{amsmath}
\usepackage{fullpage}
\usepackage{xspace}
\usepackage[psamsfonts]{amssymb}
\usepackage[latin1]{inputenc}
\usepackage{graphicx,color}
\usepackage[curve]{xypic} 
\usepackage{hyperref}
\usepackage{graphicx}

\usepackage{amsmath}%
\usepackage{amsthm}%
\usepackage{amscd}
\usepackage{amsfonts}%
\usepackage{amssymb}%
\usepackage{graphicx}

\usepackage{mathrsfs}

\usepackage{tikz}
\usetikzlibrary{matrix,arrows}

\usepackage{tikz-cd}

\newtheorem{theorem}{Theorem}[section]

\newtheorem{conjecture}[theorem]{Conjecture}
\newtheorem{corollary}[theorem]{Corollary}

\newtheorem{lemma}[theorem]{Lemma}

\newtheorem{proposition}[theorem]{Proposition}

\theoremstyle{remark}

\numberwithin{equation}{section}

\newcommand{\Acal}{\mathscr{A}}

\newcommand{\Ecal}{\mathscr{E}}
\newcommand{\Fcal}{\mathscr{F}}

\newcommand{\Lcal}{\mathscr{L}}

\newcommand{\Ocal}{\mathscr{O}}

\newcommand{\Pro}{\mathbb{P}}

\newcommand{\C}{\mathbb{C}}

\newcommand{\Proj}{\mathrm{Proj}}

\newcommand{\adeg}{\mathrm{adeg}}
\newcommand{\cdeg}{\mathrm{cdeg}}

\newcommand{\Ram}{\mathrm{Ram}}
\newcommand{\dd}{\mathrm{d}}

\makeatletter
\@namedef{subjclassname@2020}{%
  \textup{2020} Mathematics Subject Classification}
\makeatother

  \DeclareFontFamily{U}{wncy}{}
    \DeclareFontShape{U}{wncy}{m}{n}{<->wncyr10}{}
    \DeclareSymbolFont{mcy}{U}{wncy}{m}{n}
    \DeclareMathSymbol{\Sha}{\mathord}{mcy}{"58}

\begin{document}
\title[]{Algebroid maps and hyperbolicity of symmetric powers}

\author{Natalia Garcia-Fritz}
\address{ Departamento de Matem\'aticas,
Pontificia Universidad Cat\'olica de Chile.
Facultad de Matem\'aticas,
4860 Av.\ Vicu\~na Mackenna,
Macul, RM, Chile}
\email[N. Garcia-Fritz]{natalia.garcia@uc.cl}%

\author{Hector Pasten}
\address{ Departamento de Matem\'aticas,
Pontificia Universidad Cat\'olica de Chile.
Facultad de Matem\'aticas,
4860 Av.\ Vicu\~na Mackenna,
Macul, RM, Chile}
\email[H. Pasten]{hector.pasten@uc.cl}%

\thanks{N.G.-F. was supported by ANID Fondecyt Regular grant 1211004 from Chile. H.P. was supported by ANID Fondecyt Regular grant 1230507 from Chile.}
\date{\today}

\begin{abstract} Given a complex projective algebraic variety $X$ we define $ h(X)$ as the largest $n$ such that the $n$-th symmetric power of $X$ is (Brody) hyperbolic. Using Nevanlinna theory for algebroid maps, we give non-trivial lower bounds for $ h(X)$. From an arithmetic point of view, the problem is closely related to the finiteness of algebraic points of bounded degree in varieties over number fields. We provide explicit applications of our results in the case of curves embedded in surfaces and in the case of subvarieties of abelian varieties. 
\end{abstract}

\maketitle



\section{Introduction} 

\subsection{Hyperbolicity and symmetric powers} The goal of this article is to develop tools for establishing hyperbolicity of symmetric powers of a complex projective variety. This subject is motivated by considerations in algebraic geometry, complex geometry, and arithmetic, which will be discussed below. 

The topic of hyperbolicity of symmetric powers $X^{(n)}$ of a variety $X$ was first studied in a systematic way by Cadorel, Campana, and Rousseau \cite{CCR} using tools from algebraic geometry and the fundamental vanishing theorem of Demailly \cite{Dem}  and Siu--Yeung \cite{SY} for pull-backs of jet differentials via holomorphic functions, applied to the case of holomorphic maps into $X^{(n)}$. In our approach we prove instead a vanishing theorem for algebroid functions (Theorem \ref{ThmAlgHyp}) which will allow us to directly work on $X$ rather than on $X^{(n)}$. In this way we obtain hyperbolicity criteria (Theorems \ref{ThmAlgHyp} and  \ref{ThmSymmHyp}) with explicit numerical conditions that will allow us to apply our results in concrete examples. Before describing our results in a precise way, we need to discuss some preliminaries.

\subsection{Some notation} 
Let $X$ be a complex projective variety. We say that $X$ is (Brody) \emph{hyperbolic} if every holomorphic map $f\colon\C\to X$ is constant.

For $n\ge 0$ we define the \emph{$n$-th symmetric power} of a variety $X$ as $X^{(n)}=X^n/S_n$, where the symmetric group $S_n$ acts on $X^n$ by permuting the coordinates (in particular, $X^{(0)}$ is a point.) 

Let us define the \emph{hyperbolicity level} $ h(X)$ of the variety $X$ as the largest integer $n\ge 0$ such that $X^{(n)}$ is hyperbolic. We note that $ h(X)$ is well-defined (finite) which can be seen by considering a generically finite map $X\dashrightarrow \Pro^{\dim X}$ and pulling back a general line ---this actually shows that $h(X)$ is bounded by the degree of irrationality of $X$. Also, we note that $X^{(m)}$ is hyperbolic if and only if $m\le  h(X)$; indeed, for such an $m$ one can embed $X^{(m)}$ into $X^{(h(X))}$. In particular, $ h(X)>0$ means that $X$ is hyperbolic.
 The main problem that we consider is to give non-trivial lower bounds for $h(X)$. We will approach this problem by means of Nevanlinna theory of algebroid functions. 

Let us recall that, roughly speaking, an algebroid function of degree $d$ can be thought as a $d$-valued holomorphic map $f\colon\C\to X$. Then one sees that for $n\ge 1$, the variety $X^{(n)}$ is hyperbolic if and only if every algebroid map into $X$ of degree $d\le n$ is constant. See Section \ref{SecAlgebroidIntro} for details.

\subsection{Connections with arithmetic} Besides the intrinsic interest of the subject of hyperbolicity, its study is also motivated by arithmetic considerations. Indeed, we have the following conjecture of Lang (cf.\ \cite{LangHyp0, LangHyp}):

\begin{conjecture}[Lang] \label{ConjLang} Let $Y$ be a projective variety defined over a number field $k$. If $Y_\C$ is hyperbolic, then $Y(L)$ is finite for every number field $L$ extending $k$.
\end{conjecture}

Lang's conjecture was proved by Faltings when $Y$ is contained in an abelian variety.

Regarding hyperbolicity of symmetric powers, there is also a remarkable connection with arithmetic. First, we have the following standard proposition:

\begin{proposition} Let $X$ be a projective variety defined over a number field $k$ and let $n\ge 1$. If for every number field $L$ extending $k$ the set $X^{(n)}(L)$ is finite, then the following holds: For every number field $L$ extending $k$, the variety $X$ contains at most finitely many algebraic points of degree $\le n$ over $L$. \end{proposition}

Hence we have:

\begin{corollary} Let $X$ be a projective variety defined over a number field $k$ and let $n\ge 1$. If $X^{(n)}_\C$ is hyperbolic and if Lang's conjecture \ref{ConjLang} holds for $X^{(n)}$, then the following holds: For every number field $L$ extending $k$, the variety $X$ contains at most finitely many algebraic points of degree $\le n$ over $L$.
\end{corollary}

Lang's conjecture is not needed in the case of curves, thanks to results of Faltings \cite{Faltings2}, see for instance \cite{Frey,  HarrisSilverman, Vogt2, Vogt1, VojtaQuadratic} and the references therein. In this context, let us make two remarks on the hyperbolicity level $ h(C)$ of a smooth projective curve $C$:

\begin{itemize} 

\item Suppose that $C$ is defined over a number field $k$. The \emph{arithmetic degree of irrationality} ${\rm a.irr}_{\overline{k}}(C)$ as defined in \cite{Vogt1} satisfies ${\rm a.irr}_{\overline{k}}(C) =  h(C)+1$. This follows by considering the addition map $C^{(e)}\to {\rm Jac}(C)$ and applying the results of Faltings \cite{Faltings2} and Kawamata \cite{Kawamata}.

\item If $C$ is defined over $\C$, then the gonality $\gamma(C)$ of $C$ satisfies $\gamma(C)\ge  h(C)+1\ge \gamma(C)/2$. This follows from Proposition 1 (ii) in \cite{Frey}.

\end{itemize}

\subsection{Algebroid functions} \label{SecAlgebroidIntro}

Classically, it was understood that an algebroid map $w\colon\C\to X$ of degree $d$ is a $d$-valued holomorphic function. For instance $f(z)=\sqrt{z}$ defines an algebroid function of degree $2$ into $\Pro^1$. Another point of view on algebroid functions is given by considering algebraic elements over the field of complex meromorphic functions in one variable (such as $\sqrt{z}$.) However, for our treatment it will be more convenient to follow an alternative approach. 

An \emph{algebroid function to $X$} is a triple $(B,f,\pi)$ where
\begin{itemize}
\item $B$ is a connected Riemann surface,
\item $\pi\colon B\to \C$ is a proper, surjective, ramified covering of finite degree, and
\item $f\colon B\to X$ is a holomorphic map.
\end{itemize}
The \emph{covering degree} of $(B,f,\pi)$ is defined as $\cdeg (B,f,\pi)=\deg \pi$. On the other hand, for very general values of $z\in \C$, the set
$$
V_z=\{f(b) : \pi(b)=z\}
$$
always has the same cardinality $d$, and we define the \emph{algebroid degree} of $(B,f,\pi)$ as $\adeg(B,f,\pi)=d$. Note that $\cdeg (B,f,\pi)\ge \adeg (B,f,\pi)$, and it is not difficult to show that there is an algebroid map $(C,g,\nu)$ to the same $X$ with $\cdeg (C,g,\nu)= \adeg (C,g,\nu)$ such that there is a proper ramified cover $p\colon B\to C$ with $f=g\circ p$ and $\pi=\nu\circ p$. The existence of $(C,g,\nu)$ follows the standard Galois correspondence between finite coverings of Riemann surfaces and finite algebraic extensions of their fields of meromorphic functions. We refer the reader to \cite{Valiron,Ullrich,Selberg} for the beginning of the theory of value distribution of algebroid functions. See \cite{Stoll} for a standard reference on the subject.

When $\pi\colon B\to \C$ is understood from the context, we simply write $f$ instead of $(B,f,\pi)$.

Algebroid functions and hyperbolicity of $X^{(n)}$ are related by the following fact (see Section \ref{SecHyp}):

\begin{proposition}\label{ProphIntro} Let $X$ be a projective variety over $\C$ and $n\ge 1$ be an integer. The following are equivalent:
\begin{itemize}
\item[(i)] The hyperbolicity level of $X$ satisfies $ h(X)\ge n$
\item[(ii)] $X^{(n)}$ is hyperbolic
\item[(iii)] Every algebroid map $f\colon B\to X$ with $\adeg(f) \le n$ is constant (note that all possible $(B,f,\pi)$ must be considered.)
\end{itemize}
\end{proposition}

To establish cases of (i) and (ii) we will approach (iii) using Nevanlinna theory for algebroid maps; the goal is to establish (iii) with an $n$ as large as possible.

\subsection{Main results}

As a first step towards our hyperbolicity results, we prove the following vanishing theorem for pull-backs of symmetric differentials under an algebroid map.

\begin{theorem}[Vanishing]\label{ThmVanishing} Let $X$ be a smooth complex projective variety, let $\Acal$ be a very ample line sheaf on $X$, let $a$ and $m$ be positive integers, and let $\omega\in H^0(X, (\Acal^\vee)^{\otimes a}\otimes S^m\Omega^1_X)$. Let $f\colon B\to X$ be an algebroid map and let $d=\adeg(f)$. If $2(d-1) < a/m$ then $f^*\omega=0$.
\end{theorem}

Before discussing our main hyperbolicity results we need some notation.
For a vector sheaf $\Ecal$ on a smooth complex projective variety $X$ we consider the projectivized vector sheaf $P=\Proj_X \Ecal$ which comes with a canonical projection $p:P\to X$ and a tautological line sheaf $\Ocal_P(1)$ associated to $\Ecal$. For every line sheaf $\Lcal$ on $X$ we have an isomorphism $P\simeq \Proj_X(\Lcal\otimes \Ecal)$ and we have that for every $m$ there is a canonical isomorphism 
$$
H^0(X,\Lcal\otimes S^m\Ecal)\simeq H^0(P, p^*\Lcal\otimes \Ocal_P(m)).
$$

We now specialize to the case $\Ecal=\Omega^1_X$, where write $P=\Proj_X \Omega^1_X$ which comes with the canonical projection $p:P\to X$. Given a line sheaf $\Lcal$ on a smooth complex projective variety $X$, let us introduce the invariant $\sigma(X, \Lcal)$ by setting
$$
\sigma(X, \Lcal)=\sup\{a/m : a\ge 0, m\ge 1\mbox{ and }(p^*\Lcal^{\vee})^{\otimes a}\otimes \Ocal_P(m) \mbox{ is globally generated on $P$}\}
$$
and we define $\sigma(X, \Lcal)=0$ if $(p^*\Lcal^{\vee})^{\otimes a}\otimes \Ocal_P(m)$ is never globally generated  on $P$. One can define a more general invariant by replacing $\Omega^1_X$ by other vector sheaves, but this will not be necessary for us. We make the following remarks:
\begin{itemize}
\item The condition that the line sheaf $(p^*\Lcal^{\vee})^{\otimes a}\otimes \Ocal_P(m)$ be globally generated on $P$ is weaker than the condition that the vector sheaf $(\Lcal^{\vee})^{\otimes a}\otimes S^m\Omega^1_X$ be globally generated on $X$. 
\item If $\Omega^1_X$ is ample (which, by definition, means that $\Ocal_P(1)$ is an ample line sheaf on $P$) then for every line sheaf $\Lcal$ there are $a,m>0$ such that $(p^*\Lcal^{\vee})^{\otimes a}\otimes \Ocal_P(m)$ is globally generated on $P$. Thus, if $\Omega^1_X$ is ample, for every line sheaf $\Lcal$ on $X$ we have $\sigma(X,\Lcal)>0$ (possibly $\sigma(X,\Lcal)=+\infty$, for instance, when $\Lcal=\Ocal_X$.) 
\end{itemize}

We see $\sigma(X, \Lcal)$ as a measure of positivity of $\Omega^1_X$ with respect to $\Lcal$. For more information on base loci and positivity notions of vector bundles we refer the reader to \cite{BaseLoci}.

Our main hyperbolicity result for algebroid maps is (cf.\ Section \ref{SecHyp}):

\begin{theorem}[Criterion for non-existence of algebroid maps] \label{ThmAlgHyp} Let $X$ be a smooth complex projective variety, let $\Acal$ be a very ample line sheaf on $X$, and let $n\ge 1$ be an integer satisfying
$$
2(n-1) < \sigma(X, \Acal).
$$
Then every algebroid map $f\colon B\to X$ with $\adeg(f)\le n$ is constant.
\end{theorem}

Using Proposition \ref{ProphIntro}, we will deduce from Theorem \ref{ThmAlgHyp} our main hyperbolicity result for symmetric powers (cf.\ Section \ref{SecHyp}):

\begin{theorem}[Hyperbolicity of symmetric powers] \label{ThmSymmHyp} Let $X$ be a smooth complex projective variety and let $\Acal$ be a very ample line sheaf on $X$. Then $h(X)\ge \sigma(X,\Acal)/2$.
\end{theorem}

We remark that Cadorel, Campana, and Rousseau have a related result for pseudo-hyperbolicity using jet differentials instead of the sheaves of symmetric differentials $S^m\Omega^1_X$, see Theorem 4 in \cite{CCR}. In the notation of Theorem \ref{ThmAlgHyp}, their numerical condition reads $2n(n-1) < \sigma(X, \Acal)$. So, our result improves on this numerical restriction.

\subsection{Applications}\label{SecAppl}

As a first application we have the following lower bound for $ h(C)$ when $C$ is a smooth curve in a surface.

\begin{theorem}[The case of curves in surfaces]\label{ThmCurves} Let $S$ be a smooth complex projective surface and let $H$ be a very ample divisor on $S$. Let $\ell \ge 1$ be an integer and let $C\subseteq S$ be a smooth projective curve with $C\sim \ell H$. Then
$$
 h(C)\ge \frac{1}{2}\left(\ell + \frac{H.K_S}{H^2}\right)
$$
where $K_S$ is the canonical class of $S$.
\end{theorem}

Let us finally discuss higher dimensional varieties. In this case we give an explicit construction for subvarieties of abelian varieties (not obtained as products of curves) large values of $\sigma(X,\Acal)$ in a fixed dimension.

\begin{theorem}\label{ThmAbVar} Let $A$ be a complex abelian variety, let $\Acal$ be a symmetric very ample line sheaf on $A$, let $X\subseteq A$ be a positive-dimensional smooth complete intersection of ample smooth hypersurfaces in $A$, and assume that $\sigma(X,\Acal|_X)>0$ (this is the case, for instance, if $\Omega^1_X$ is ample.) For an integer $\ell \ge 1$, let $X_\ell  = [\ell]^*X$ where $[\ell]\colon A\to A$ is the multiplication-by-$\ell$ map. Then $X_\ell$ is a smooth projective variety with
$$
2 h(X_\ell)\ge \sigma(X_\ell,\Acal|_{X_\ell}) \ge \ell^2\cdot \sigma(X,\Acal|_X).
$$
\end{theorem}

Varieties $X$ satisfying the assumptions of the theorem are abundant; see \cite{Debarre}. See also \cite{Brotbek}.


\section{Nevanlinna theory and a ramification bound}

\subsection{Basic notation} We will use the standard functions appearing in Nevanlinna theory for algebroid maps, namely, the proximity function, the counting function, and the characteristic (or height) function; see for instance Section 27 in \cite{VojtaCIME}. In particular, for an algebroid map $(B,f,\pi)$ all these functions are normalized by the covering degree $\cdeg(B,f,\pi)=\deg \pi$. Let us fix the notation for these Nevanlinna-theoretical functions.

Let $X$ be a smooth complex projective variety, let $D$ be a divisor on $X$, and let $f\colon B\to X$ be an algebroid map to $X$ (with $\pi\colon B\to \C$) whose image is not contained in the support of $D$. Let us denote by $r$ a positive real variable. The proximity function relative to $D$ is $m_f(D,r)$, the counting function relative to $D$ is $N_f(D,r)$, and the height relative to $D$ is $T_f(D,r)$. These functions are well-defined up to adding a bounded function $O(1)$ and, as usual, they are normalized by $\cdeg(f)=\deg(\pi)$.

The fact that these functions are normalized by the covering degree is useful due to the following observation: If $p\colon C\to B$ is a proper finite surjective map of Riemann surfaces and $g=f\circ p\colon C\to X$, then $m_g(D,r)=m_f(D,r)$ and $N_g(D,r)=N_f(D,r)$. 

The First Main Theorem of Nevanlinna theory says that if $D'$ is another divisor on $X$ which is linearly equivalent to $D$ and whose support does not contain the image of $f$, then $T_f(D',r)=T_f(D,r)+O(1)$. Thus, one can define (up to adding $O(1)$) the height function of $f$ relative to any line sheaf $\Lcal$, which we denote by $T_f(\Lcal,r)$. We refer the reader to Section 27 in \cite{VojtaCIME}.

We also need the counting function for the ramification of $\pi\colon B\to \C$. For a point $b\in B$ we let $e_b$ be the ramification multiplicity of $\pi$ at $b$ and note that $e_b=1$ except for a discrete countable set of points. Let us define
$$
n_{\Ram(\pi)}(t)=\sum_{|\pi(b)|\le t} (e_b-1).
$$
With this notation, the ramification counting function is (cf.\ \cite{Ullrich} and \cite{VojtaCIME})
$$
N_{\Ram(\pi)}(r)=\frac{1}{\deg(\pi)}\left(n_{\Ram(\pi)}(0)\log r + \int_{0}^r \left(n_{\Ram(\pi)}(t)-n_{\Ram(\pi)}(0)\right)\frac{\dd t}{t}\right).
$$

\subsection{The ramification bound} We will need the following upper bound for $N_{\Ram(\pi)}(r)$ in terms of the height of $f$. 

\begin{theorem}[Ramification bound]\label{ThmRamBd} Let $X$ be a smooth projective variety over $\C$, let $\Acal$ be a very ample line sheaf on $X$, and let $(B,f,\pi)$ be an algebroid map to $X$. Suppose that $\adeg(f)=\cdeg(f)$ and let $d$ be this number. Then 
$$
N_{\Ram(\pi)}(r) \le 2(d-1)T_f(\Acal,r) + O(1).
$$
\end{theorem}
This is a generalization of the following classical result of Ullrich \cite{Ullrich}:

\begin{theorem}[Ullrich]\label{ThmUllrich} Let $(B,f,\pi)$ be an algebroid map to $\Pro^1$ satisfying that $\adeg(f)=\cdeg(f)$, and let $d$ be this number. Then 
$$
N_{\Ram(\pi)}(r) \le 2(d-1)T_f(\Ocal_{\Pro^1}(1),r) + O(1).
$$
\end{theorem}

 We remark that Theorem \ref{ThmRamBd} is analogous to Silverman's bound relating the height and the logarithmic discriminant of an algebraic point in projective space (cf.\ Theorem 2 in \cite{SilvermanDisc}.)

As we will explain, Theorem \ref{ThmRamBd} can be proved using general results of \cite{Stoll} (which extends \cite{Noguchi}), but we will give a simpler proof. For this we need the following observation:
\begin{lemma} \label{LemmaLifting} Let $p\colon Y\to X$ be a birational morphism of smooth complex projective varieties, let $Z$ be a proper Zariski closed set of $X$ such that $p$ is an isomorphism outside $Z$, and let $f\colon B\to X$ be an algebroid map whose image is not contained in $Z$. Then there is a unique holomorphic lifting $\widetilde{f}\colon B\to Y$ of $f$; thus $f=p\circ \widetilde{f}$.
\end{lemma}
\begin{proof} This is a consequence of the valuative criterion of properness applied to $\C[[z]]$.
\end{proof}

With this at hand, we can prove Theorem \ref{ThmRamBd}.

\begin{proof}[Proof of Theorem \ref{ThmRamBd}] It suffices to prove the result in the case $X=\Pro^M$ and $\Acal=\Ocal_{\Pro^M}(1)$; the general case follows by considering an embedding $\theta\colon X\to \Pro^M$ associated to the very ample line sheaf $\Acal$ and applying the functoriality property of the Nevanlinna height:
$$
T_f(\Acal,r)=T_f(\theta^*\Ocal_{\Pro^M}(1),r) + O(1)=T_{\theta\circ f} (\Ocal_{\Pro^M}(1),r) + O(1).
$$
Further, by Ullrich's Theorem \ref{ThmUllrich} we may assume $M\ge 2$. We remark that the case $X=\Pro^M$ and $\Acal=\Ocal_{\Pro^M}(1)$ is a consequence of Corollary 6.8 in \cite{Stoll}, but for the convenience of the reader we will give a shorter proof that suffices in our setting.

Then we have an algebroid function $f\colon B\to \Pro^M$ with $\adeg(f)=\deg(\pi)=d$. Let $q\colon \Pro^M\dashrightarrow \Pro^1$ be a linear projection with indeterminacy locus $L$ (a linear subspace of codimension $2$), which is general enough so that the following holds:
\begin{itemize}
\item The image of $f$ is not contained in $L$ or in the hyperplane $H:=q^*[0:1]\subseteq \Pro^M$, and 
\item $g:= q\circ f$ extends to an algebroid function $g\colon B\to \Pro^1$ with $\adeg(g)=\adeg(f)=\deg(\pi)=d$ (this is possible because given $d$ different points in $\Pro^M$, a general choice of $q$ will map them injectively to $\Pro^1$.)
\end{itemize}

Let $\beta \colon Y\to \Pro^M$ be the blow up along $L$ with exceptional divisor $E$, and let $\widetilde{q}\colon Y\to \Pro^1$ be the induced morphism. By Lemma \ref{LemmaLifting} we obtain an algebroid function $\widetilde{f}\colon B\to Y$ such that  $f=\beta\circ \widetilde{f}$.  The following commutative diagram summarizes the maps that we are studying:

$$
 \begin{tikzcd}
   & & Y \arrow{ddrr}{\widetilde{q}} \arrow{d}{\beta} & & \\
     & & \Pro^M \arrow[dashed, swap]{drr}{q} & & \\
   B\arrow{uurr}{\widetilde{f}} \arrow{rrrr}{g} \arrow[swap]{urr}{f}& & &  & \Pro^1
 \end{tikzcd}
 $$

Let us  consider $D=[0:1]$ as a divisor on $\Pro^1$, so that $H=q^*D$ is a hyperplane in $\Pro^M$. We remark that
\begin{equation}\label{EqnDivPB}
\widetilde{q}^*D \le \beta^*q^*D = \beta^*H
\end{equation}
and the inequality is strict, since $E$ is not in the support of $\widetilde{q}^*D$ because $E$ surjects onto $\Pro^1$ via $\widetilde{q}$.

By Ullrich's Theorem \ref{ThmUllrich} and the fact that $\adeg(g)=\adeg(f)=\deg(\pi)=d$ we have
$$
N_{\Ram(\pi)}(r)\le 2(d-1)T_g(\Ocal_{\Pro^1}(1), r) +O(1)=2(d-1)T_g(D, r) +O(1).
$$
Now, using functoriality of the Nevanlinna height we get
$$
T_{g}(D,r)= T_{\widetilde{q}\circ \widetilde{f}}(D,r)=T_{\widetilde{f}}(\widetilde{q}^*D,r) + O(1).
$$
By positivity of the Nevanlinna height in the case of effective divisors and by \eqref{EqnDivPB}, we obtain
$$
T_{\widetilde{f}}(\widetilde{q}^*D,r) \le T_{\widetilde{f}}(\beta^*H,r) +O(1)= T_{\beta\circ \widetilde{f}}(H,r)+O(1)=T_f(\Ocal_{\Pro^M}(1),r)+O(1).
$$
Putting the last three computations together we obtain the desired result.
\end{proof}


\section{Proof of the main results}


\subsection{Vanishing} If an inequality holds for the positive real variable $r$ outside a set of finite measure, we use the symbol $\le_{\rm exc}$. Let us recall the following corollary of McQuillan's Tautological Inequality (Theorem A in \cite{McQuillan}), see Corollary 29.9 in \cite{VojtaCIME} with $D=0$.

\begin{theorem}[Vanishing using a ramification term]\label{ThmMcQuillan} Let $X$ be a smooth complex projective variety, let $\Acal$ be an ample line sheaf on $X$, let $\Lcal$ be a line sheaf on $X$, let $m\ge 1$ be an integer and let $\omega\in H^0(X, \Lcal^{\vee}\otimes S^m\Omega^1_X)$. Let $(B,f,\pi)$ be an algebroid map to $X$. If $f^*\omega$ is not identically zero on $B$, then 
$$
\frac{1}{m} T_f(\Lcal,r)\le_{\rm exc} N_{\Ram(\pi)}(r) + o\left(T_{f}(\Acal ,r)\right).
$$
\end{theorem}

Strictly speaking, Corollary 29.9 in \cite{VojtaCIME} only gives this when $f$ is transcendental, due to the error term. But when $f$ is algebraic this directly follows from the Riemann--Hurwitz formula as follows:

Let $C$ be a smooth projective algebraic curve, let $\pi:C\to \Pro^1$ be surjective, and let $f:C\to X$. Keep the notation and assumptions of Theorem \ref{ThmMcQuillan} regarding $\omega$, in particular $f^*\omega\ne 0$ on $C$. Then $f^*\Lcal^\vee \otimes S^m\Omega^1_C$ has a non-zero section and therefore it has non-negative degree on $C$. This gives 
$$
\frac{1}{m}\deg_C( f^*\Lcal)  \le \deg_C(\Omega^1_C) = 2(g(C)-1)
$$
where $g(C)$ is the genus of $C$. Then one applies the Riemann--Hurwitz formula on $\pi:C\to \Pro^1$ to get the ramification term from $2(g(C)-1)$. This concludes the argument in the algebraic case. 

Let us now prove our vanishing result.

\begin{proof}[Proof of Theorem \ref{ThmVanishing}] Without loss of generality we may assume that $\cdeg(f)=\adeg(f)=d$. Then the result is proved by applying Theorem \ref{ThmMcQuillan} with our choice of $\Acal$ very ample, with $\Lcal=\Acal^{\otimes a}$, and then the result follows from our Theorem \ref{ThmRamBd} which allows us to control $N_{\Ram(\pi)}(r)$. 

Namely, for the sake of contradiction assume $f^*\omega\ne 0$; in particular $f$ is non-constant. Then with the previous choices we get
$$
\frac{a}{m}T_f(\Acal,r)\le_{\rm exc} 2(d-1)T_f(\Acal,r) + o\left(T_{f}(\Acal ,r)\right).
$$
Since $2(d-1)<a/m$, we reach a contradiction by letting $r\to \infty$.
\end{proof}


\subsection{Hyperbolicity}\label{SecHyp} Let us record the following observation:

\begin{lemma}\label{LemmaConst} Let $X$ be a smooth complex projective variety of dimension $n$, let $P=\Proj_X(\Omega^1_X)$ with canonical projection $\pi:P\to X$, let $\Lcal$ be a line sheaf on $X$, and let $m\ge 1$ be an integer. Suppose that $\pi^*\Lcal\otimes \Ocal_P(m)$ is globally generated on $P$. If $f\colon B\to X$ is an algebroid map with $f^*\omega=0$ for each $\omega\in H^0(X,\Lcal\otimes S^m\Omega^1_X)$, then $f$ is constant.
\end{lemma}
\begin{proof} For the sake of contradiction, assume that $f$ is non-constant and let $f':B\to P$ be its canonical lift. If $f$ is $\omega$-integral for every $\omega\in H^0(X,\Lcal\otimes S^m\Omega^1_X)$, then the image of $f'$ in $P$ is contained in the base locus of the line sheaf $\pi^*\Lcal\otimes \Ocal_P(m)$, but this is empty because $\pi^*\Lcal\otimes \Ocal_P(m)$ is globally generated. So, it is not possible that $f$ be non-constant; this is a contradiction.
\end{proof}

We now return to our hyperbolicity results:

\begin{proof}[Proof of Theorem \ref{ThmAlgHyp}] We let $P=\Proj_X(\Omega^1_X)$ with canonical projection $p:P\to X$. Let us choose $a$ and $m$ such that  $2(n-1)< a/m$ and $(p^*\Acal^\vee)^{\otimes a}\otimes \Ocal_P(m)$ is globally generated on $P$; this is possible by the definition of $\sigma(X, \Acal)$ and the condition $2(n-1)< \sigma(X, \Acal)$. By Lemma \ref{LemmaConst} with $\Lcal= (\Acal^\vee)^{\otimes a}$, the result follows from Theorem \ref{ThmVanishing}.
\end{proof}

To connect with hyperbolicity of symmetric powers we now prove Proposition \ref{ProphIntro}:

\begin{proof}[Proof of Proposition \ref{ProphIntro}] It suffices to show that (iii) implies (ii), because the equivalence of (i) and (ii) as well as the implication (ii)$\Rightarrow$(iii) are clear. (Also, note that  we don't really need (ii)$\Rightarrow$(iii) in our discussion; we study algebroid maps to establish hyperbolicity of symmetric powers which corresponds to the implication (iii)$\Rightarrow$(ii).)

Suppose that (ii) fails. Consider the quotient map $\rho\colon X^n\to X^{(n)}$ and let $f\colon\C\to X^{(n)}$ a non-constant holomorphic map (which exists because (ii) fails.) Composing $f$ with $\rho^{-1}$ we get a non-constant multivalued holomorphic function to $X^n$, which therefore can be written as $F\colon\sqcup_{j} B_j\to X^n$ where each $B_j$ is a connected Riemann surface with a finite proper surjective map $\pi_j:B_j\to \C$ and  $F|_{B_j}\to X^n$ is an algebroid map. We take one of them which is non-constant, say $F_1=F|_{B_1}\colon B_1\to X^n$. Letting $w\colon B_1\to X$ be the first projection $X^n\to X$ composed with $F_1$, we see that $w$ is a non-constant algebroid map. 

Furthermore, for any given $z\in \C$ we have that $f(z)=\{x_1,...,x_n\}$ as a multiset with some $x_j\in X$. Therefore, if $b\in B_1$ satisfies $\pi_1(b)=z$ we see that $w(b)\in \{x_1,...,x_n\}$. Thus,  $\adeg(w)\le n$.
\end{proof}

Finally, let us prove our main hyperbolicity result for symmetric powers.

\begin{proof}[Proof of Theorem \ref{ThmSymmHyp}] By Theorem \ref{ThmAlgHyp} there is no non-constant algebroid map $f\colon B\to X$ with $2(d-1)<\sigma(X,\Acal)$ where $d=\adeg(f)$. Proposition \ref{ProphIntro} then shows that $ h(X)\ge n$ where $n$ is the largest integer with $2(n-1) <\sigma(X,\Acal)$, and this implies $ h(X)\ge \sigma(X,\Acal)/2$.
\end{proof}


\section{Applications}


\subsection{Curves in surfaces}

\begin{proof}[Proof of Theorem \ref{ThmCurves}] Let $K_C$ be the canonical divisor class of $C$. Then, by adjunction we get
$$
\deg K_C = C^2 + C.K_S = \ell^2H^2 + \ell H.K_S.
$$
Let $A=H|_C$; this is a very ample divisor on $C$. We let $\Acal=\Ocal_C(A)$. Note that $\deg A= C.H=\ell H^2$. 

We observe that $\sigma(C,\Acal)$ is at least equal to the supremum of the numbers $a/m$ as $a,m\ge 1$ vary over positive integers such that $-aA + mK_C$ has positive degree. Thus, $\sigma(C,\Acal)\ge \tau$ where $\tau$ is the supremum of the positive real numbers $t$ with
$$
-t \ell H^2 + \ell^2H^2 + \ell H.K_S>0.
$$
As $H$ is ample, we have $H^2>0$. Thus we see that the previous condition is equivalent to
$$
t< \ell + (H.K_S)/H^2.
$$
Hence, $\sigma(C,\Acal)\ge \tau = \ell + (H.K_S)/H^2$. The result now follows from Theorem \ref{ThmSymmHyp}.
\end{proof}


\subsection{Subvarieties of abelian varieties}

\begin{proof}[Proof of Theorem \ref{ThmAbVar}] Since $X$ is a complete intersection of ample smooth hypersurfaces and $[\ell]\colon A\to A$ is finite, $X_\ell$ is also a complete intersection of ample smooth hypersurfaces. Therefore $X_\ell$ is connected (inductively apply Corollary 7.9 in p.244 of \cite{Hartshorne}.) On the other hand, since $X$ is smooth and $[\ell]$ is \'etale, we have that $X_\ell$ is smooth. Therefore, $X_\ell$ is irreducible. It will be convenient to write $\Acal_\ell=\Acal|_{X_\ell}$, which is a very ample line sheaf on $X_\ell$.

Since $\Acal$ is symmetric (i.e. $[-1]^*\Acal\simeq \Acal$), we have that $[\ell]^*\Acal \simeq \Acal^{\otimes \ell^2}$ by a well-known formula of Mumford, see for instance Corollary A.7.2.5 in \cite{HindrySilverman}. 

Let us write $P_\ell = \Proj_{X_\ell} \Omega^1_{X_\ell}$ which comes with the canonical projection $p_\ell:P_\ell\to X_\ell$. 

Let $a,m\ge 1$ be integers with $(p^*\Acal_1^{\vee})^{\otimes a}\otimes  \Ocal_{P_1}(m)$ globally generated on $P_1$; this is possible because $\sigma(X,\Acal)>0$. Let $\rho_\ell= [\ell]|_{X_\ell}\colon X_\ell\to X$ the multiplication-by-$\ell$ map restricted to $X_\ell$ and by Mumford's formula
$$
\rho_\ell^*\left((\Acal_1^{\vee})^{\otimes a}\otimes  S^m\Omega^1_X\right)\simeq 
\left((([\ell]^*\Acal)|_{X_\ell})^{\vee}\right)^{\otimes a}\otimes S^m\rho_\ell^*\Omega^1_{X} \simeq (\Acal_\ell^{\vee})^{\otimes a\ell^2}\otimes  S^m\Omega^1_{X_\ell}
$$
where the isomorphism $\rho_\ell^*\Omega^1_{X}\simeq\Omega^1_{X_\ell}$ is due to the fact that $\rho_\ell$ is finite, surjective, and \'etale.

Let $\overline{\rho_\ell}:P_\ell\to P_1$ be the map induced by $\rho_\ell$. The previous isomorphism induces an isomorphism
$$
\overline{\rho_\ell}^*\left((p_1^*\Acal_1)^{\otimes a}\otimes \Ocal_{X_1}(m)\right) \simeq (p_\ell^*\Acal_\ell)^{\otimes a\ell^2}\otimes \Ocal_{X_\ell}(m).
$$

When $u:Y\to Z$ is a morphism of varieties and $\Fcal$ is a globally generated line sheaf on $Z$, one has that $u^*\Fcal$ is also globally generated on $Y$. It follows that $(p_\ell^*\Acal_\ell)^{\otimes a\ell^2}\otimes \Ocal_{X_\ell}(m)$ is globally generated on $P_\ell$ (taking $u=\overline{\rho_\ell}:P_\ell\to P_1$ and $\Fcal=(p_1^*\Acal_1)^{\otimes a}\otimes \Ocal_{X_1}(m)$), from which we obtain $\sigma(X_\ell,\Acal_\ell)\ge a\ell^2/m$. Taking the supremum of $a/m$, we obtain the result.
\end{proof}

\section{Acknowledgments}

N.G.-F. was supported by ANID Fondecyt Regular grant 1211004 from Chile.

H.P. was supported by ANID Fondecyt Regular grant 1230507 from Chile.

The first author would like to thank the Isaac Newton Institute for Mathematical Sciences, Cambridge, for support and hospitality during the programme \emph{New equivariant methods in algebraic and differential geometry}, where work on this paper was undertaken. This work was supported by EPSRC grant EP/R014604/1.


\end{document}